\documentclass[12pt,reqno]{amsart}
\usepackage{amscd,amsmath,amsthm,amssymb}
\usepackage{color}
\usepackage{pstricks}
\usepackage{stmaryrd}
\usepackage{tikz}
\usepackage{url}

\usepackage{latexsym}
\usepackage{amsfonts,amsmath,mathtools}
\usepackage{graphics}
\usepackage{float}
\usepackage{enumitem}

\newpsstyle{fatline}{linewidth=1.5pt}
\newpsstyle{fyp}{fillstyle=solid,fillcolor=verylight}
\definecolor{verylight}{gray}{0.97}
\definecolor{light}{gray}{0.9}
\definecolor{medium}{gray}{0.85}
\definecolor{dark}{gray}{0.6}

\def\NZQ{\mathbb}               % the font for N,Z,Q,R,C

\def\QQ{{\NZQ Q}}
\def\ZZ{{\NZQ Z}}

\def\CC{{\NZQ C}}

\def\G{{\mathcal G}}

\def\pd{\textup{proj}\phantom{.}\!\textup{dim}}

\def\opn#1#2{\def#1{\operatorname{#2}}} % to make operators
\opn\chara{char} \opn\length{\ell} \opn\pd{pd} \opn\rk{rk}
\opn\projdim{proj\,dim} \opn\injdim{inj\,dim} \opn\rank{rank}
\opn\depth{depth} \opn\grade{grade} \opn\height{height}
\opn\embdim{emb\,dim} \opn\codim{codim}

\opn\Tr{Tr} \opn\bigrank{big\,rank}
\opn\superheight{superheight}\opn\lcm{lcm}
\opn\trdeg{tr\,deg}%\emph{
\opn\reg{reg} \opn\lreg{lreg} \opn\ini{in} \opn\lpd{lpd}
\opn\size{size} \opn\sdepth{sdepth}
\opn\link{link}\opn\fdepth{fdepth}\opn\lex{lex}
\opn\tr{tr}
\opn\type{type}
\opn\gap{gap}
\opn\diam{diam}
\opn\Mod{Mod}
\opn\div{div} \opn\Div{Div} \opn\cl{cl} \opn\Cl{Cl}
\opn\Spec{Spec} \opn\Supp{Supp} \opn\supp{supp} \opn\Sing{Sing}
\opn\Ass{Ass} \opn\Min{Min}\opn\Mon{Mon}
\opn\Ann{Ann} \opn\Rad{Rad} \opn\Soc{Soc}
\opn\Im{Im} \opn\Ker{Ker} \opn\Coker{Coker} \opn\Am{Am}
\opn\Hom{Hom} \opn\Tor{Tor} \opn\Ext{Ext} \opn\End{End}
\opn\Aut{Aut} \opn\id{id}

\opn\nat{nat}
\opn\pff{pf}%   \pf exists already
\opn\Pf{Pf} \opn\GL{GL} \opn\SL{SL} \opn\mod{mod} \opn\ord{ord}
\opn\Gin{Gin} \opn\Hilb{Hilb}\opn\sort{sort}
\opn\PF{PF}\opn\Ap{Ap}
\opn\dist{dist}
\opn\aff{aff}
\opn\relint{relint} \opn\st{st}
\opn\lk{lk} \opn\cn{cn} \opn\core{core} \opn\vol{vol}  \opn\inp{inp} \opn\nilpot{nilpot}
\opn\link{link} \opn\star{star}\opn\lex{lex}\opn\set{set}
\opn\width{wd}
\opn\Fr{F}
\opn\QF{QF}
\opn\G{G}
\opn\type{type}\opn\res{res}
\opn\conv{conv}
\opn\sr{sr}
\opn\gr{gr}

\def\pot#1#2{#1[\kern-0.28ex[#2]\kern-0.28ex]}

\opn\dirlim{\underrightarrow{\lim}}
\opn\inivlim{\underleftarrow{\lim}}
\def\Implies{\ifmmode\Longrightarrow \else
	\unskip${}\Longrightarrow{}$\ignorespaces\fi}
\def\implies{\ifmmode\Rightarrow \else
	\unskip${}\Rightarrow{}$\ignorespaces\fi}
\def\iff{\ifmmode\Longleftrightarrow \else
	\unskip${}\Longleftrightarrow{}$\ignorespaces\fi}

\let\:=\colon
\newtheorem{Theorem}{Theorem}[section]
\newtheorem{Lemma}[Theorem]{Lemma}
\newtheorem{Corollary}[Theorem]{Corollary}
\newtheorem{Proposition}[Theorem]{Proposition}
\newtheorem{Remark}[Theorem]{Remark}

\newtheorem{Example}[Theorem]{Example}

\newtheorem{Conjecture}[Theorem]{Conjecture}

\let\epsilon\varepsilon
\let\kappa=\varkappa
\def\qed{\ifhmode\textqed\fi
	\ifmmode\ifinner\hfill\quad\qedsymbol\else\dispqed\fi\fi}
\def\textqed{\unskip\nobreak\penalty50
	\hskip2em\hbox{}\nobreak\hfill\qedsymbol
	\parfillskip=0pt \finalhyphendemerits=0}
\def\dispqed{\rlap{\qquad\qedsymbol}}

\opn\dis{dis}
\def\pnt{{\raise0.5mm\hbox{\large\bf.}}}

\opn\Lex{Lex}
\opn\Max{Max}
\opn\Shad{Shad}
\opn\astab{astab}
\def\p{\mathfrak{p}}
\def\q{\mathfrak{q}}
\def\m{\mathfrak{m}}

\opn\v{v}

\begin{document}

	\title{Simon Conjecture and\\ the $\mbox{v}$-number of monomial ideals}
	\author{Antonino Ficarra}
	%\date{\today}
	
	\address{Antonino Ficarra, Department of mathematics and computer sciences, physics and earth sciences, University of Messina, Viale Ferdinando Stagno d'Alcontres 31, 98166 Messina, Italy}
	\email{antficarra@unime.it}
	
	\thanks{
	}
	
	
	
	\maketitle
	
	\begin{abstract}
		Let $I\subset S$ be a graded ideal of a standard graded polynomial ring $S$ with coefficients in a field $K$, and let $\v(I)$ be the $\v$-number of $I$. In previous work, we showed that for any graded ideal $I\subset S$ generated in a single degree, then $\v(I^k)=\alpha(I)k+b$, for all $k\gg0$, where $\alpha(I)$ is the initial degree of $I$ and $b$ is a suitable integer. In the present paper, using polarization, we extend Simon conjecture to any monomial ideal. As a consequence, if Simon conjecture holds, and all powers of $I$ have linear quotients, then $b\in\{-1,0\}$. This fact suggest that if $I$ is an equigenerated monomial ideal with linear powers, then $\v(I^k)=\alpha(I)k-1$, for all $k\ge1$. We verify this conjecture for monomial ideals with linear powers having $\depth S/I=0$, edge ideals with linear resolution, polymatroidal ideals, and Hibi ideals.
	\end{abstract}
	
	\section{Introduction}
	
	Let $S=K[x_1,\dots,x_n]=\bigoplus_dS_d$ be the standard graded polynomial ring with $n$ variables and coefficients in a field $K$, and let $\m=(x_1,\dots,x_n)$ be the graded maximal ideal. We denote the set of associated primes of $I$ by $\Ass(I)$, and by $\Max(I)$ the set of associated primes of $I$ that are maximal with respect to the inclusion. The concept of $\v$-number was introduced in \cite{CSTVV20}, and further studied in  \cite{ASS23,BM23,Civan2023,FS2,GRV21,JS23,JV21,S2023,SS20,VPV23}. Let $I\subset S$ be a graded ideal and let $\p\in\Ass(I)$. Then, the \textit{$\v$-number of $I$ at $\p$} is defined as
	$$
	\v_\p(I)\ =\ \min\{d\ :\ \textit{there exists}\ f\in S_d\ \textit{such that}\ (I:f)=\p\}.
	$$
	Whereas, the \textit{$\v$-number of $I$} is defined as
	$$
	\v(I)\ =\ \min\{d\ :\ \textit{there exists}\ f\in S_d\ \textit{such that}\ (I:f)\in\Ass(I)\}.
	$$
	
	Let $I\subset S$ be a graded ideal. In \cite{FS2} the asymptotic behaviour of the function $\v(I^k)$, called the \textit{$\v$-function} of $I$, was investigated. It is known by Brodmann \cite{B79} that $\Ass(I^k)$ stabilizes for large $k$. That is, $\Ass(I^{k+1})=\Ass(I^k)$ for all $k\gg0$. A prime ideal $\p\subset S$ such that $\p\in\Ass(I^k)$ for all $k\gg0$, is called a \textit{stable prime of $I$}. The set of the stable primes of $I$ is denoted by $\Ass^{\infty}(I)$. Thus, for all $k\gg0$, $$\v(I^k)=\min_{\p\in\Ass^{\infty}(I)}\v_\p(I^k).$$ 
	
	It is expected that for any graded ideal $I\subset S$, $\v(I^k)$ becomes a linear function in $k$ for $k\gg0$ \cite[Conjecture 4.5]{FS2}. Such a conjecture has been proved when $I$ is generated in a single degree, when $I$ does not have embedded stable primes \cite[Theorem 3.1]{FS2}, and for all monomial ideals in two variables \cite[Theorem 5.1]{FS2}.
	
	To compute the $\v$-number, we use the following result of Grisalde, Reyes and Villarreal \cite[Theorem 3.2]{GRV21}. Let $M=\bigoplus_{d\ge0}M_d\ne0$ be a finitely generated graded $S$-module. We call $\alpha(M)=\min\{d:M_d\ne0\}$, the \textit{initial degree} of $M$, and set $\omega(M)=\max\{d:(M/\m M)_d\ne0\}$.  The bar $\overline{\phantom{p}}$ denotes the residue class modulo $I$.
	\begin{Theorem}\label{Thm:GRVtheorem}
		Let $I\subset S$ be a graded ideal and let $\p\in\Ass(I)$. The following hold.
		\begin{enumerate}
			\item[\textup{(a)}] If $\mathcal{G}=\{\overline{g_1},\dots,\overline{g_r}\}$ is a homogeneous minimal generating set of $(I:\p)/I$, then
			$$
			\v_\p(I)=\min\{\deg(g_i)\ :\ 1\le i\le r\ \textit{and}\ (I:g_i)=\p\}.
			$$
			\item[\textup{(b)}] $\v(I)=\min\{\v_\p(I):\p\in\Ass(I)\}$.
			\item[\textup{(c)}] $\v_\p(I)\ge\alpha((I:\p)/I)$, with equality if $\p\in\Max(I)$.
			\item[\textup{(d)}] If $I$ has no embedded primes, then $\v(I)=\min\{\alpha((I:\p)/I):\p\in\Ass(I)\}$.
		\end{enumerate}
	\end{Theorem}
	
	In this paper, we study the $\v$-function of monomial ideals with linear powers.
	
	The paper is structured as follows. In Section \ref{sec:F-V-Num2}, we show in Theorem \ref{Thm:v(I^k)-Equigen} that for an equigenerated graded ideal $I\subset S$, we have $\v(I^k)=\alpha(I)k+b$ for some integer $b$, and all $k\gg0$. In general, the integer $b$ depends on $\textup{char}(K)$, as we show in Example \ref{Ex:v(I^k)dependsOnChar(K)}. However, if $I$ is a monomial ideal, then $b$ does not depend on $\textup{char}(K)$ (Proposition \ref{Prop:v(I^k)MonNotDependOnChar}) and moreover $b\ge-1$ (Theorem \ref{Thm:v(I^k)-Equigen}(c)). Indeed, the lower bound $\alpha(I)k-1\le\v(I^k)$ (Proposition \ref{Prop:Vnum>alpha}) always holds. In general, we also have the lower bound $\alpha(I)k\le\reg(I^k)$ for the Castelnuovo--Mumford regularity. If $I$ has linear powers, then this lower bound is achieved for all $k$. In view of this fact, it is natural to expect that $\v(I^k)=\alpha(I)k-1$ for all $k\ge1$, when $I$ is a monomial ideal with linear powers (Conjecture \ref{Conj:v-num-lin-powers}). Next, we investigate this conjecture.
	
	A particular class of monomial ideals with linear powers is that of equigenerated monomial ideals whose all powers have linear quotients. Linear quotients is the algebraic counterpart of the concept of shellability in simplicial complexes theory. A famous conjecture of Simon \cite[Conjecture 4.2.1]{Simon94} can be expressed, in algebraic terms, by saying that any squarefree monomial ideal $I\subset S$ generated in a single degree $d$ and having linear quotients, can be extended to the squarefree Veronese ideal $I_{n,d}$ by linear quotients (Conjecture \ref{Conj:Simon}). We extend this conjecture to all monomial ideals (Conjecture \ref{Conj:MonomialSimon}). Surprisingly, by using polarization, it turns out that the conjectures are equivalent (Proposition \ref{Prop:ConseqSimon}). For another instance of such an use of polarization, see \cite[Proposition 1.8]{EreyFicarra2023}. As a consequence, if Simon conjecture holds, then for any monomial ideal generated in a single degree and whose all powers have linear quotients, either $\v(I^k)=\alpha(I)k-1$ or $\v(I^k)=\alpha(I)k$, for all $k\gg0$.
	
	The behaviour of the $\v$-number of monomial ideals under polarization was studied in \cite{SS20}. The authors showed that $\v(I^\wp)\le\v(I)$ and equality holds in some special cases. In Theorem \ref{Theorem:v(I)v(I^wp)}, surprisingly we show that the equality $\v(I)=\v(I^\wp)$ holds for any monomial ideal $I$.
	
	Finally, we prove that Conjecture \ref{Conj:v-num-lin-powers} holds: when $\depth S/I=0$ (Theorem \ref{Thm:VNumDepth=0}), for edge ideals with linear resolution (Theorem \ref{Thm:vI(G)linearRes}), for polymatroidal ideals (Theorem \ref{Thm:vNumPolym}) and for Hibi ideals (Theorem \ref{Thm:vNumHibi}).

	\section{The $\v$-function of monomial ideals}\label{sec:F-V-Num2}
	
	In this section, we summarize some basic facts about the $\v$-function of monomial ideals. Hereafter, for a positive integer $n$, we set $[n]=\{1,2,\dots,n\}$. If $I\subset S$ is a monomial ideal, then each associated prime $\p\in\Ass(I)$ is a \textit{monomial prime ideal}. That is, $\p=\p_A=(x_i:i\in A)$ for some nonempty subset $A$ of $[n]$.
	
	The following result \cite[Proposition 3.11]{SS20}, due to Saha and Sengupta, is a special case of \cite[Proposition 4.3]{FS2}. It provides an useful general method to bound $\v(I)$ from above, when $I$ is a monomial ideal.
	\begin{Proposition}\label{Prop:SahaSengupta}
		Let $I\subset S$ be a monomial ideal and $f\in S\setminus I$ a monomial. Then,
		$$
		\v(I)\le\v(I:f)+\deg(f).
		$$
	\end{Proposition}
	
	On the other hand, one always has
	\begin{Proposition}\label{Prop:Vnum>alpha}
		Let $I\subset S$ be a monomial ideal. Then,
		$$
		\v_\p(I)\ge\alpha(I)-1,\ \ \textit{for all}\ \ \p\in\Ass(I).
		$$
		In particular, $\v(I^k)\ge\alpha(I)k-1$, for all $k\ge1$.
	\end{Proposition}
	\begin{proof}
		Let $\p\in\Ass(I)$ and let $u\in S$ be a monomial such that $(I:u)=\p$ and $\deg(u)=\v_\p(I)$. Then $\p=\p_A=(x_i:i\in A)$ for some $A\subseteq[n]$. Thus $x_iu\in I$ for all $i\in A$. In particular, $\deg(x_iu)\ge\alpha(I)$. Hence, $\deg(u)\ge\alpha(I)-1$, as desired.
	\end{proof}
	
	From the work of \cite{FS2}, the next result follows.
	\begin{Theorem}\label{Thm:v(I^k)-Equigen}
		Let $I\subset S$ be a graded ideal generated in a single degree.
		\begin{enumerate}
			\item[\textup{(a)}] For all $\p\in\Ass^\infty(I)$, there exists an integer $b_\p$ such that $\v_\p(I^k)=\alpha(I)k+b_\p$ is a linear function in $k$, for all $k\gg0$.
			\item[\textup{(b)}] $\v(I^k)=\alpha(I)k+b$ for all $k\gg0$, where $b=\min_{\p\in\Ass^\infty(I)}b_\p$.
			\item[\textup{(c)}] If $I$ is a monomial ideal, then $b_\p,b\ge-1$, for all $\p\in\Ass^\infty(I)$.
		\end{enumerate}
	\end{Theorem}
	\begin{proof}
		The proof of \cite[Theorem 3.1(b)]{FS2} shows that $\v_\p(I^k)=\alpha(I)k+b_\p$, for all $k\gg0$ and all $\p\in\Ass^\infty(I)$, for some $b_\p\in\ZZ$. Setting $b=\min\limits_{\p\in\Ass^{\infty}(I)}b_\p$, then
		$$
		\v(I^k)\ =\ \min_{\p\in\Ass^{\infty}(I)}\v_\p(I^k)\ =\ \min_{\p\in\Ass^{\infty}(I)}(\alpha(I)k+b_\p)\ =\ \alpha(I)k+b,
		$$
		for all $k\gg0$, and (a) and (b) follow. Part (c) follows from Proposition \ref{Prop:Vnum>alpha}.
	\end{proof}
	
	Recall that a graded ideal $I\subset S$ has a \textit{$d$-linear resolution} if it is generated in a single degree $d$, and for all $i\ge0$, $\beta_{i,j}(I)=0$ if $j\ne i+d$. Equivalently, an equigenerated ideal $I$ has linear resolution, if and only if, the Castelnuovo--Mumford regularity is as small as possible, namely $\reg(I)=\alpha(I)$. We say that $I$ has \textit{linear powers} if $I^k$ has a linear resolution, for all $k\ge1$. Famous examples of monomial ideals with linear powers are given in the following list.\smallskip
	\begin{enumerate}\label{list:MonLinPow}
		\item[(i)] Edge ideals with linear resolution \cite[Theorem 10.2.6]{JT}.
		\item[(ii)] Polymatroidal ideals \cite[Corollary 12.6.4]{JT}.
		\item[(iii)] Hibi ideals \cite[Corollary 10.2.9 and Theorem 9.1.13]{JT}, or \cite[Corollary 4.11]{CF2023}.
	\end{enumerate}\medskip
	
	In general, the $\v$-function of a graded ideal $I$ depends on the characteristic of $K$, even when $I$ has linear powers.
	\begin{Example}\label{Ex:v(I^k)dependsOnChar(K)}
		\rm Indeed, consider the principal ideal $I=(x^2+y^2)$ of $S=\QQ[x,y]$. Then $I^k$ is principal as well, and thus $I$ has linear powers. Since $x^2+y^2$ is an irreducible polynomial over $\QQ[x,y]$, $I$ is a prime ideal, and thus $I^k$ is a $I$-primary ideal and $\Ass(I^k)=\{I\}$ for all $k$. Note that
		$$
		(I^k:I)/I^k\ =\ (((x^2+y^2)^k):(x^2+y^2))/I^k\ =\ ((x^2+y^2)^{k-1})/I^k
		$$
		for all $k\ge1$. Hence, $\v(I^k)=\v_I(I^k)=2k-2$ for all $k\ge1$.
		
		On the other hand, if $I=(x^2+y^2)$ is an ideal of $\CC[x,y]$, then $(x^2+y^2)^k=(x+iy)^k(x-iy)^k$ for all $k\ge1$. Setting $\p_1=(x+iy)$ and $\p_2=(x-iy)$, then $I^k=\p_1^k\p_2^k=\p_1^k\cap\p_2^k$ is the primary decomposition of $I^k$ and $\Ass(I^k)=\{\p_1,\p_2\}$ for all $k\ge1$. Note that
		$$
		(I^k:\p_1)/I^k\ =\ (((x^2+y^2)^k):(x+iy))/I^k\ =\ ((x+iy)^{k-1}(x-iy)^k)/I^k
		$$
		for all $k\ge1$. Hence $\v_{\p_1}(I^k)=2k-1$ for all $k\ge1$. By symmetry, we also have $\v_{\p_2}(I^k)=2k-1$ for all $k\ge1$. Thus $\v(I^k)=2k-1$ for all $k\ge1$.
	\end{Example}\medskip
	
	However, the primary decomposition of a monomial ideal $I\subset S$ does not depend on the characteristic, $\textup{char}(K)$, of the field $K$ \cite[Theorem 1.3]{JT}. Next, since each $\p\in\Ass^\infty(I)$ is a monomial prime ideal, it follows from \cite[Proposition 1.2.2]{JT} that the monomial $K$-basis of $(I^k:\p)/I^k$ is independent from $\text{char}(K)$. This observations show that
	\begin{Proposition}\label{Prop:v(I^k)MonNotDependOnChar}
		Let $I\subset S$ be a monomial ideal. Then, the functions $\v_\p(I^k)$ for each $\p\in\Ass^\infty(I)$ and all $k\gg0$, and the function $\v(I^k)$ for $k\ge1$, do not depend upon $\textup{char}(K)$.
	\end{Proposition}

	Now, suppose that $I\subset S$ is a monomial ideal with linear powers. Then, the regularity of $I^k$ attains its minimum value possible, namely $\alpha(I^k)=k\alpha(I)$. Hence, $\reg(I^k)=\alpha(I)k$, for all $k\ge1$. Thus it is natural to expect that the $\v$-number of $I^k$ attains its minimal possible value for all $k$, namely $\v(I^k)=\alpha(I)k-1$ for all $k\ge1$.
	\begin{Conjecture}\label{Conj:v-num-lin-powers}
		Let $I\subset S$ be a monomial ideal with linear powers. Then
		$$
		\v(I^k)\ =\ \alpha(I)k-1, \ \ \text{for all}\ k\ge1.
		$$
	\end{Conjecture}\medskip
	
	If $I\subset S$ is a graded ideal with linear powers, but not a monomial ideal, then the above conjecture is false, as the ideal $(x^2+y^2)\subset\QQ[x,y]$ of Example \ref{Ex:v(I^k)dependsOnChar(K)} shows.
	
	Similarly, if $I$ is a monomial ideal, but does not have linear powers, Conjecture \ref{Conj:v-num-lin-powers} is no longer valid. Next example is due to Terai \cite[Remark 3]{C2000}. If $\textup{char}(K)\ne 2$, the Stanley Reisner ideal $I = (abd, abf, ace, adc, aef, bde, bcf, bce, cdf, def)$ of the minimal triangulation of the projective plane has a linear resolution, but $I^2$ has not. By using \textit{Macaulay2} \cite{GDS}, we have $\v(I)=\alpha(I)=3$ and $\v(I^k)=3k-1$ for all $k\ge2$.
	
	\section{The monomial Simon conjecture}\label{sec:F-V-Num3}
	
	Let $I\subset S$ be a monomial ideal and $G(I)$ be the minimal monomial generating set of $I$. We say that $I$ has \textit{linear quotients} if there exists an order $u_1,\dots,u_m$ of $G(I)$ such that $(u_1,\dots,u_{j-1}):u_j$ is generated by variables, for $j=2,\dots,m$.
	
	The monomial ideal of $S$ generated by all squarefree monomial ideals of degree $1\le d\le n$ is called the \textit{Veronese ideal} of degree $d$ of $S$, and it is denoted by $I_{n,d}$. For instance, $I_{4,3}=(x_1x_2x_3,x_1x_2x_4,x_1x_3x_4,x_2x_3x_4)$.
	
	Let $I\subset J$ be monomial ideals with $G(I)\subset G(J)$. We say that $I$ can be \textit{extended to $J$ by linear quotients} if the set $G(J)\setminus G(I)$ can be ordered as $v_1,\dots,v_r$ such that $(I,v_1,\dots,v_{j-1}):v_j$ is generated by variables, for all $j=2,\dots,r$. 
	
	The famous Simon conjecture \cite[Conjecture 4.2.1]{Simon94} states that the skeleton of any simplex is \textit{extendably shellable}, see \cite{BYPZN2019} and the references therein for more details on this topic. In algebraic terms, the conjecture can be equivalently stated as follows.
	\begin{Conjecture}\label{Conj:Simon}
		Let $I\subset S$ be a squarefree monomial ideal with linear quotients such that $G(I)\subset G(I_{n,d})$. Then $I$ can be extended to $I_{n,d}$ by linear quotients.
	\end{Conjecture}

    It is natural to extend the above conjecture to all monomial ideals.
    \begin{Conjecture}\label{Conj:MonomialSimon}
    	Let $I\subset S$ be a monomial ideal with linear quotients such that $G(I)\subset G(\mathfrak{m}^d)$. Then $I$ can be extended to $\mathfrak{m}^d$ by linear quotients.
    \end{Conjecture}
    
    We refer to the above conjecture as the \textit{monomial Simon conjecture}. Since the squarefree part of an ideal with linear quotients has linear quotients \cite[Lemma 3.1]{EHHM2022b}, this latter conjecture implies the usual Simon conjecture. Surprisingly, however, it turns out that the two statements are equivalent, as we show next.
    
	\begin{Proposition}\label{Prop:SimonVsSimon}
		Simon conjecture is equivalent to the monomial Simon conjecture.
	\end{Proposition}
	
	We use \textit{polarization}. Let $u=x_1^{b_1}\cdots x_n^{b_n}\in S$ be a monomial. The \textit{$x_i$-degree} of $u$ is the integer $\deg_{x_i}(u)=b_i$ for all $i$. The \textit{polarization} of $u$ is the monomial
	$$
	u^\wp\ =\ \prod_{i=1}^n(\prod_{j=1}^{b_i}x_{i,j})\ =\ \prod_{\substack{1\le i\le n\\ b_i>0}}x_{i,1}x_{i,2}\cdots x_{i,b_i}
	$$
	in the polynomial ring $K[x_{i,j}:i\in[n],j\in[b_i]]$.\smallskip

	Let $I\subset S$ be a monomial ideal. Following \cite{F2},  the \textit{bounding multidegree} of $I$ is the vector ${\bf deg}(I)=(\deg_{x_1}(I),\dots,\deg_{x_n}(I))$, with $\deg_{x_i}(I)=\max_{u\in G(I)}\deg_{x_i}(u)$, for all $1\le i\le n$. The \textit{polarization} of $I$ is defined to be the squarefree ideal $I^\wp$ of $S^\wp=K[x_{i,j}:i\in[n],j\in[\deg_{x_i}(I)]]$ such that $G(I^\wp)=\{u^\wp:u\in G(I)\}$.\smallskip
	
	 Attached to the polarization of $I$, we have the specialization map $\pi:S^\wp\rightarrow S$ defined by setting $\pi(x_{i,j})=x_i$ for all $i$ and $j$.
	\begin{proof}[Proof of Proposition \ref{Prop:SimonVsSimon}]
		We only need to show that the Simon conjecture implies the monomial version. Let $I\subset S$ be a monomial ideal generated in degree $d$ with linear quotients. Say with order $u_1,\dots,u_m$ such that $(u_1,\dots,u_{j-1}):u_j$ is generated by variables, for $j=2,\dots,m$. Then, \cite[Lemma 3.3]{Jahan2007} (see, also, \cite[Lemma 4.10]{CF2023}) implies that $I^\wp$ also has linear quotients with order $u_1^\wp,\dots,u_m^\wp$. Denote by $J$ the Veronese ideal of degree $d$ of $S^\wp$. By our assumption, $I^\wp$ can be extended to $J$ by linear quotients. Thus, there is an order $u_1^\wp,\dots,u_m^\wp,v_1,\dots,v_r$ of $G(J)$ such that $J$ has linear quotients with respect to this order. Applying the specialization map, we obtain an order $u_1,\dots,u_m,\pi(v_1),\dots,\pi(v_r)$ of generators of $\mathfrak{m}^d$ such that $(u_1,\dots,u_{j-1}):u_j$, for $j=2,\dots,m$, and $(u_1,\dots,u_m,\pi(v_1),\dots,\pi(v_{j-1})):\pi(v_j)$, for $j=2,\dots,r$, are generated by variables. Removing all minimal generators of $\mathfrak{m}^d$ in this order wherever they appear again, we obtain an order of the minimal generators of $\mathfrak{m}^d$ whose beginning is $u_1,\dots,u_m$ such that $\mathfrak{m}^d$ has linear quotients with respect to this order. Hence, $I$ can be extended to $\mathfrak{m}^d$ by linear quotients, and the monomial Simon conjecture holds. 
	\end{proof}

    As a consequence, we have the following result.
    \begin{Proposition}\label{Prop:ConseqSimon}
    	Assume the monomial Simon conjecture. Let $I\subset S$ be a monomial ideal generated in a single degree and having linear quotients. Then
    	$$
    	\alpha(I)-1\le\v(I)\le\alpha(I).
    	$$
    \end{Proposition}
    \begin{proof}
    	Let $\alpha(I)=d$. If $I=\m^d$, then $\Ass(I)=\{\m\}$, $(I:\m)/I=\m^{d-1}/\m^d$ and so $\v(I)=\v_{\m}(I)=d-1=\alpha(I)-1$. Suppose now that $I\ne\m^d$. Then, $I$ can be extended to $\m^d$ by linear quotients. Hence there exists a monomial $u\in\m^d\setminus I$ such that $(I:u)$ is generated by variables. Applying Proposition \ref{Prop:SahaSengupta}, we get
    	$$
    	\v(I)\ \le\ \v(I:u)+\deg(u)=0+d=\alpha(I).
    	$$
    	This inequality together with Proposition \ref{Prop:Vnum>alpha} yields the conclusion.
    \end{proof}

    The next result supports Conjecture \ref{Conj:v-num-lin-powers}.
    \begin{Corollary}
    	Assume the monomial Simon conjecture. Let $I\subset S$ be a monomial ideal generated in single degree whose all powers have linear quotients. Then, either $\v(I^k)=\alpha(I)k-1$ or $\v(I^k)=\alpha(I)k$, for all $k\gg0$.
    \end{Corollary}
    \begin{proof}
    	By Theorem \ref{Thm:v(I^k)-Equigen}(c), $\v(I^k)=\alpha(I)k+b$ for all $k\gg0$, with $b\ge-1$. On the other hand, Proposition \ref{Prop:ConseqSimon} implies that $b\le0$. Thus either $b=-1$ or $b=0$.
    \end{proof}

    We end the section, with the next general result which supports Conjecture \ref{Conj:v-num-lin-powers}.
    \begin{Theorem}\label{Thm:VNumDepth=0}
    	Let $I\subset S$ be a monomial ideal with linear powers, such that $\depth S/I=0$. Then,
    	$$
    	\v(I^k)=\alpha(I)k-1,\ \ \textit{for all}\ k\ge1.
    	$$
    \end{Theorem}
    \begin{proof}
    	Let $\alpha(I)=d$. By \cite[Proposition 10.3.4]{JT}, $\depth(S/I^k)$ is a nonincreasing function. Hence, $\depth S/I^k=0$ for all $k\ge1$. Therefore, $\m\in\Ass(I^k)$ for all $k\ge1$.
    	
    	On the other hand, by \cite[Lemma 3.1]{F2}, the condition $\depth S/I^k=0$ is equivalent to $G(I^k:\m)_{dk-1}\ne0$, for all $k\ge1$. Here $G(J)_j$ denotes the set of all monomials of degree $j$ belonging to $G(J)$. Since $I^k$ is generated in degree $dk$, it follows that $\alpha((I^k:\m)/I^k)=dk-1$ for all $k\ge1$. Hence, Theorem \ref{Thm:GRVtheorem}(c) implies that $\v_\m(I^k)=\alpha((I^k:\m)/I^k)=dk-1$ for all $k\ge1$. Thus $\v(I^k)\le\v_\m(I^k)=\alpha(I)k-1$ for all $k\ge1$. Equality follows from Proposition \ref{Prop:Vnum>alpha}.
    \end{proof}

	\section{Reduction to the squarefree case}\label{sec:F-V-Num4}
	
	In \cite{SS20}, the authors studied the behaviour of the $\v$-number of monomial ideals under polarization. They showed that $\v(I^\wp)\le\v(I)$ and equality holds in some special cases. In this section, surprisingly we show that the equality $\v(I)=\v(I^\wp)$ holds for all monomial ideals.
	
	As in the previous section, let $I\subset S$ be a monomial ideal, and let $I^\wp$ be the polarization of $I$ in the polynomial ring $S^\wp=K[x_{i,j}:i\in[n],j\in[\deg_{x_i}(I)]]$.
	\begin{Theorem}\label{Theorem:v(I)v(I^wp)}
		Let $I\subset S$ be a monomial ideal. Then,
		\begin{enumerate}
			\item[\textup{(a)}] $\p\in\Ass(I)$ if and only if there exists $\q\in\Ass(I^\wp)$ with $\pi(\q)=\p$.
			\item[\textup{(b)}] For all $\q\in\Ass(I^\wp)$, we have $\v_\p(I)\le\v_\q(I^\wp)$, where $\p=\pi(\q)$.
			\item[\textup{(c)}] For all $\p\in\Ass(I)$, we have $$\v_\p(I)=\min\{\v_q(I^\wp)\ :\ \q\in\Ass(I^\wp),\pi(\q)=\p\}.$$
			\item[\textup{(d)}] $\v(I)=\v(I^\wp)$.
		\end{enumerate}
	\end{Theorem}
	\begin{proof}
		(a) is proved in \cite[Corollary 3.7]{Faridi2006}. 
		
		(b) Let $\q\in\Ass(I^\wp)$ and let $\p=\pi(\q)$ be, thanks to part (a), the corresponding prime of $\Ass(I)$. Suppose that $f\in S^\wp$ is a monomial such that $\v_\q(I^\wp)=\deg(f)$ and $(I^\wp:f)=\q$. In view of \cite[Proposition 1.2.2]{JT} it is clear that $\p=\pi(\q)=\pi(I^\wp:f)=(I:\pi(f))$. Thus, $\v_\p(I)\le\deg(\pi(f))=\deg(f)=\v_\q(I^\wp)$.
		
		(c) Let $f\in S$ be a monomial such that $(I:f)=\p$ and $\deg(f)=\v_\p(I)$. Then, by \cite[Lemma 3.3]{Jahan2007} we have that $(I^\wp:f^\wp)$ is prime and $\p=(I:f)=\pi((I^\wp:f^\wp))$. Thus $(I^\wp:f^\wp)=\q\in\Ass(I^\wp)$. This shows $\v_\q(I^\wp)\le\v_\p(I)$ and together with part (b) we have $\v_\p(I)=\v_\q(I^\wp)$. Since $\v_\p(I)\le\v_{\mathfrak{r}}(I^\wp)$ for all $\mathfrak{r}\in\Ass(I^\wp)$ such that $\pi(\mathfrak{r})=\p$, it follows that $\v_\q(I)$ and so, also $\v_\p(I)$, is equal to $\min\{\v_{\mathfrak{r}}(I):\mathfrak{r}\in\Ass(I^\wp),\pi(\mathfrak{r})=\p\}$.
		
		(d) By part (c) we obtain that
		$$
		\v(I)\ =\ \min_{\p\in\Ass(I)}\v_\p(I)\ =\ \min_{\p\in\Ass(I)}(\min_{\substack{\q\in\Ass(I^\wp)\\ \pi(\q)=\p}}\v_\q(I^\wp)).
		$$
		By part (a), we conclude that the above minimum is $\v(I^\wp)$, as desired.
	\end{proof}
	
    \section{The $\v$-number of monomial ideals with linear powers}\label{sec:F-V-Num5}
    
    In this section, we study equigenerated monomial ideals $I$ arising from combinatorial contexts. In particular, we show that for any ideal $I$ in the list (i)-(ii)-(iii) at page \pageref{list:MonLinPow}, Conjecture \ref{Conj:v-num-lin-powers} holds true.
    
    \subsection{Edge ideals with linear resolution}
    
    Let $G$ be a finite simple graph with vertex set $V(G)=[n]$ and edge set $E(G)$. The \textit{edge ideal} of $G$ is the monomial ideal $I(G)$ of $S$ generated by the monomials $x_ix_j$ such that $\{i,j\}\in E(G)$. A graph $G$ is \textit{complete} if every $\{i,j\}$ with $i,j\in[n]$, $i\ne j$, is an edge of $G$. The \textit{open neighbourhood} of $i\in V(G)$ is the set
    $$
    N_G(i)=\{j\in V(G):\{i,j\}\in E(G)\}.
    $$
    Whereas, the \textit{closed neighbourhood} of $i\in V(G)$ is defined as $N_G[i]=N_G(i)\cup\{i\}$.
    
    A graph $G$ is called \textit{chordal} if it has no induced cycles of length bigger than three. Recall that a \textit{perfect elimination order} of $G$ is an ordering $v_1,\dots,v_n$ of its vertex set $V(G)$ such that $N_{G_i}(v_i)$ induces a complete subgraph on $G_i$, where $G_i$ is the induced subgraph of $G$ on the vertex set $\{i,i+1,\dots,n\}$. Hereafter, if $1,2,\dots,n$ is a perfect elimination order of $G$, we denote it by $x_1>x_2>\dots>x_n$.
    
    A famous theorem of Dirac guarantees that a finite simple graph $G$ is chordal if and only if $G$ admits a perfect elimination order \cite{Dirac61}.
    
    The \textit{complementary graph} $G^c$ of $G$ is the graph with vertex set $V(G^c)=V(G)$ and where $\{i,j\}$ is an edge of $G^c$ if and only if $\{i,j\}\notin E(G)$. A graph $G$ is called \textit{cochordal} if and only if $G^c$ is chordal. In \cite{Froberg88}, Fr\"oberg proved that $I(G)$ has a linear resolution if and only if $G$ is cochordal.
    
    By Theorem \ref{Thm:v(I^k)-Equigen} and Proposition \ref{Prop:v(I^k)MonNotDependOnChar}, for any graph $G$ and all $k\gg0$, we have
    $$
    \v(I(G)^k)\ =\ 2k+b(G),
    $$
    where $b(G)\ge-1$ is a constant independent from $\textup{char}(K)$. In particular, $\v(I(G)^k)\ge 2k-1$ for all $k\ge1$. As a consequence of Dirac and Fr\"oberg theorems, we show that for any edge ideal with linear resolution this lower bound is achieved.
    \begin{Theorem}\label{Thm:vI(G)linearRes}
    	Let $I(G)$ be the edge ideal of a graph $G$. Suppose that $I(G)$ has a linear resolution. Then,
    	$$
    	\v(I(G)^k)\ =\ 2k-1,\ \ \textit{for all}\ \ k\ge1.
    	$$
    \end{Theorem}
    
    The proof is based upon the next lemma. If $A$ is a subset of $V(G)$, the induced subgraph of $G$ on $A$, is the graph on vertex set $A$ and edges $\{i,j\}\in E(G)$ such that $i,j\in A$.
    \begin{Lemma}\label{Lemma:I(G)colon}
    	Let $I(G)$ be an edge ideal with linear resolution, and let $x_1>x_2>\cdots>x_n$ be a perfect elimination order of $G^c$. Then,
    	\begin{equation}\label{eq:I(G):x_1}
    		(I(G):x_1)\ =\ (x_j:j\in N_G(1)).
    	\end{equation}
    \end{Lemma}
    \begin{proof}
    	It is clear that
    	$$
    	(x_j:j\in N_G(1))\ \subseteq\ (I(G):x_1).
    	$$
    	To end the proof, we show the opposite inclusion. Let $x_kx_\ell\in I(G)$ and suppose that both $k$ and $\ell$ are not in $N_G(1)$. Then $\{1,k\},\{1,\ell\}\in E(G^c)$, that is $k,\ell\in N_{G^c}(1)$. Since, $x_1>x_2>\dots>x_n$ is a perfect elimination order of $G^c$, it follows that that $N_{G^c}(1)$ induces a complete subgraph of $G_2^c$, where $G_2^c$ is the induced subgraph of $G^c$ on the vertex set $\{2,\dots,n\}$. Since $k,\ell>1$, it follows that $\{k,\ell\}\in E(G^c)$, in contradiction with $\{k,\ell\}\in E(G)$. Thus either $k\in N_{G}(1)$ or $\ell\in N_{G}(1)$ and formula (\ref{eq:I(G):x_1}) follows.
    \end{proof}
    
    As a consequence, we obtain by a different method the next result already proved in \cite[Proposition 3.19]{JV21}.
    \begin{Corollary}
    	Let $I(G)$ be the edge ideal of a graph $G$. Suppose that $I(G)$ has a linear resolution. Then, $\v(I(G))=1$.
    \end{Corollary}
    \begin{proof}
    	We proceed by induction on $|V(G)|\ge2$ with the base case being trivial. Let $|V(G)|>2$. Let $x_1>x_2>\cdots>x_n$ be a perfect elimination order of $G^c$. Then, by Lemma \ref{Lemma:I(G)colon}, equation (\ref{eq:I(G):x_1}) holds. Thus, by Proposition \ref{Prop:SahaSengupta}, $\v(I(G))\le\v(I(G):x_1)+\deg(x_1)=0+1=1$. On the other hand, by Proposition \ref{Prop:Vnum>alpha}, $\v(I(G))\ge\alpha(I(G))-1=1$. The assertion follows.
    \end{proof}
    \begin{Remark}\label{Rem:useful}
    	\rm Let $I\subset S$ be a graded ideal. Suppose that $(I^{k+1}:I)=I^k$ and $\p\in\Ass(I^k)$ for all $k\ge1$. Then, the proof of \cite[Proposition 3.1]{FS2} shows that $\v_\p(I^{k+1})\le\v_\p(I^k)+\omega(I)$ for all $k\ge1$.
    \end{Remark}
    
    Now, we are in the position to prove Theorem \ref{Thm:vI(G)linearRes}. 
    \begin{proof}[Proof of Theorem \ref{Thm:vI(G)linearRes}]
    	By the previous result, $\v(I(G))=1$. Therefore, for some $\p\in\Ass(I(G))$, we have $\v_\p(I(G))=1$. By \cite[Theorem 2.15]{MMV}, we have
    	$$
    	\Ass(I(G))\subseteq\Ass(I(G)^2)\subseteq\dots\subseteq\Ass(I(G)^k)\subseteq\cdots.
    	$$
    	Hence, $\p\in\Ass(I(G)^k)$ for all $k\ge1$. By \cite[Lemma 2.12]{MMV}, $(I(G)^{k+1}:I(G))=I(G)^k$ for all $k\ge1$. Since $\alpha(I(G)^{k+1})=2(k+1)$ and $\omega(I(G))=2$, by Remark \ref{Rem:useful} and Proposition \ref{Prop:Vnum>alpha}, we have
    	$$
    	2(k+1)-1\le\v_\p(I(G)^{k+1})\le\v_\p(I(G)^k)+2
    	$$
    	for all $k\ge1$. By induction on $k\ge1$, we may assume that $\v_\p(I(G)^{k})=2k-1$. The above chain of inequalities gives $\v_\p(I(G)^{k+1})=2(k+1)-1=\alpha(I(G)^{k+1})-1$. By Proposition \ref{Prop:Vnum>alpha} it follows that $\v(I(G)^k)=2k-1$ for all $k\ge1$, as well.
    \end{proof}

    \subsection{Polymatroidal ideals} Let $I\subset S$ be a monomial ideal generated in a single degree. Then $I$ is called \textit{polymatroidal} if the following exchange property holds: for all $u,v\in G(I)$ and all $i$ such that $\deg_{x_i}(u)>\deg_{x_i}(v)$ there exists $j$ such that $\deg_{x_j}(u)<\deg_{x_j}(v)$ and $x_j(u/x_i)\in G(I)$.
    
    Such a name is justified by the fact that the set of the multidegrees of the minimal generators of $I$ is the set of the bases of a \textit{discrete polymatroid} \cite[Chapter 12]{JT}.\medskip
    
    This class of monomial ideals is very rich. Indeed, it includes
    \begin{enumerate}
    	\item[(i)] \textit{Graphic matroids.} They are the ideals generated by the monomials $\prod_{i\in F}x_i$, for all spanning forests $F$ of a finite simple graph $G$ on $n$ vertices.
    	\item[(ii)] \textit{Transversal polymatroidal ideals.} They are of the form $I=\p_{A_1}\p_{A_2}\cdots\p_{A_r}$, for some finite collection $A_1,\dots,A_r$ of arbitrary nonempty subsets of $[n]$.
    	\item[(iii)] \textit{Ideals of Veronese type.} Let ${\bf c}=(c_1,\dots,c_n)\in\ZZ^n$ be a vector with non negative entries. Then, the ideal of \textit{Veronese type $(d,{\bf c})$} is defined as
    	$$
    	I_{n,d,{\bf c}}\ =\ (x_1^{a_1}\cdots x_n^{a_n}\in S\ :\ \sum_{i=1}^na_i=d,\ \ a_i\le c_i,\ \text{for}\ i\in[n]).
    	$$
    \end{enumerate}

    Polymatroidal ideals also satisfies a dual exchange property \cite[Lemma 2.1]{HH2003}, namely: for all $u,v\in G(I)$ and all $i$ such that $\deg_{x_i}(u)<\deg_{x_i}(v)$ there exists $j$ such that $\deg_{x_j}(u)>\deg_{x_j}(v)$ and $x_i(u/x_j)\in G(I)$.
    
    \begin{Theorem}\label{Thm:vNumPolym}
    	Let $I\subset S$ be a polymatroidal ideal. Then
    	$$
    	\v(I^k)=\alpha(I)k-1,\ \ \textit{for all}\ \ k\ge1.
    	$$
    \end{Theorem}
    
    The proof is based upon the next lemma.
    \begin{Lemma}
    	Let $I\subset S$ be a polymatroidal ideal generated in degree $\alpha(I)\ge2$. Then $(I:x_i)$ is again a polymatroidal ideal generated in degree $\alpha(I)-1$, for all $i\in[n]$.
    \end{Lemma}
    \begin{proof}
    	We may assume that $i=1$. We can write $I=x_1I_1+I_2$, where $I_1$ and $I_2$ are the unique monomial ideals of $S$ such that
    	\begin{align*}
    		G(x_1I_1)\ &=\ \{u\in G(I):x_1\ \text{divides}\ u\},\\
    		G(I_2)\ &=\ \{u\in G(I):x_1\ \text{does not divide}\ u\}.
    	\end{align*}
    	We claim that $I_2\subset I_1$. It is enough to show that $G(I_2)\subseteq I_1$. Let $u\in G(I_2)$ and let $v\in G(x_1I_1)$. Then $\deg_{x_1}(u)=0<\deg_{x_1}(v)$. Thus, by the dual exchange property, we can find $j$ such that $\deg_{x_j}(u)>\deg_{x_j}(v)$ and $x_1(u/x_j)\in G(I)$. Hence $x_1(u/x_j)\in x_1I_1$, and so $u/x_j\in I_1$. Consequently, $u\in I_1$ too, and thus $I_2\subset I_1$.
    	
    	By the previous paragraph, we have $(I:x_1)=I_1+I_2=I_1$. It is clear that $I_1$ is equigenerated in degree $\alpha(I)-1$. It remains to prove that $I_1$ is polymatroidal. Let $u_1,v_1\in G(I_1)$ and $i$ such that $\deg_{x_i}(u_1)>\deg_{x_i}(v_1)$. Our job is to find $j$ such that $\deg_{x_j}(u_1)<\deg_{x_j}(v_1)$ and $x_j(u_1/x_i)\in G(I_1)$. Set $u=x_1u_1$ and $v=x_1v_1$. Then $u,v\in G(x_1I_1)\subset G(I)$ and $\deg_{x_i}(u)>\deg_{x_i}(v)$. Since $I$ is polymatroidal, there exists $j$ such that $\deg_{x_j}(u)<\deg_{x_j}(v)$ and $x_j(u/x_i)\in G(I)$. We claim that $x_1$ divides $x_j(u/x_i)$. Indeed $x_1$ divides $u$. If $i\ne 1$, then $x_1$ divides $x_j(u/x_i)$ as well. If $i=1$, since $x_1$ divides $v$ and $\deg_{x_i}(u)>\deg_{x_i}(v)>0$, it follows that $x_1^2$ divides $u$ and so $x_1$ divides $x_j(u/x_i)=x_j(u/x_1)$. Therefore, in any case $x_1$ divides $x_j(u/x_i)$. Hence, $(x_j(u/x_i))/x_1=x_j(u_1/x_i)\in G(I_1)$ and the proof is complete.
    \end{proof}
    We are ready for the proof of the theorem.
    \begin{proof}[Proof of Theorem \ref{Thm:vNumPolym}]
    	Firstly, we show that $\v(I)=\alpha(I)-1$. We proceed by strong induction on $\alpha(I)\ge1$. If $\alpha(I)=1$, then $I=\p_A$ for some $A\subseteq[n]$, $\alpha(I)=1$ and $\v(I)=\v_{\p_A}(I)=0$. Suppose $\alpha(I)>1$. By the previous proposition, $(I:x_1)$ is a polymatroidal ideal and $\alpha(I:x_1)=\alpha(I)-1$. By induction hypothesis, $\v(I:x_1)=\alpha(I:x_1)-1=\alpha(I)-2$. Hence, by Proposition \ref{Prop:SahaSengupta},
    	$$
    	\v(I)\le\v(I:x_1)+\deg(x_1)=\alpha(I)-1.
    	$$
    	By Proposition \ref{Prop:Vnum>alpha}, $\v(I)\ge\alpha(I)-1$. Equality follows.
    	
    	Let $k>1$. It is well--known that the product of polymatroidal ideals is polymatroidal \cite[Theorem 12.6.3]{JT}. Hence, $I^k$ is a polymatroidal ideal generated in degree $\alpha(I)k$. By what shown above, $\v(I^k)=\alpha(I^k)-1=\alpha(I)k-1$.
    \end{proof}
    
    \subsection{Hibi ideals} Let $(P,\succeq)$ be a finite partially ordered set (a \textit{poset}, for short) with $P=\{p_1, \ldots, p_n\}$. A \textit{poset ideal} of $P$ is a subset $\mathcal{I}$ of $P$ such that if $p_i\in P$, $p_j\in\mathcal{I}$ and $p_i\preceq p_j$, then $p_i\in\mathcal{I}$. To any poset ideal $\mathcal{I}$ of $P$, we associate the monomial 
    $u_\mathcal{I}=(\prod_{p_i\in\mathcal{I}}x_i)(\prod_{p_i\in P\setminus\mathcal{I}}y_i)\in S=K[x_1, \ldots, x_n, y_1, \ldots, y_n]$. The set of all poset ideals of $P$ is denoted by $\mathcal{J}(P)$. Then the \textit{Hibi ideal} of $P$ is the monomial ideal of $S$ defined as
    $$
    H_P\ =\ (u_\mathcal{I}\ :\ \mathcal{I}\in\mathcal{J}(P)).
    $$
    
    \begin{Example}\label{ex:posetVNum}
    	\rm Consider the poset $(P,\succeq)$ with $P=\{p_1,p_2,p_3\}$,
    	$p_1\prec x_2$ and $p_1\prec p_3$. The poset $(P,\succeq)$ and the distributive lattice $\mathcal{J}(P)$ are depicted below:
    	
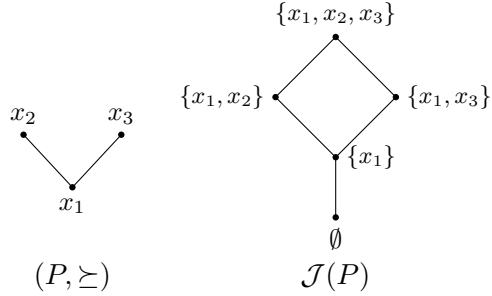
\begin{center}
    		\begin{tikzpicture}[scale=0.5]
    			\filldraw (1.3,0) circle (2pt) node[below] {\footnotesize$x_1$};
    			\filldraw (0,1.4) circle (2pt) node[above] {\footnotesize$x_2$};
    			\filldraw (2.6,1.4) circle (2pt) node[above] {\footnotesize$x_3$};
    			\draw[-] (0,1.4) -- (1.3,0) -- (2.6,1.4);
    			\filldraw (6.7,2.4) circle (2pt) node[left] {\scriptsize$\{x_1,x_2\}$};
    			\filldraw (8.3,0.8) circle (2pt) node[right] {\scriptsize$\{x_1\}$};
    			\filldraw (8.3,4) circle (2pt) node[above] {\scriptsize$\{x_1,x_2,x_3\}$};
    			\filldraw (9.9,2.4) circle (2pt) node[right] {\scriptsize$\{x_1,x_3\}$};
    			\filldraw (8.3,-0.8) circle (2pt) node[below] {\footnotesize$\emptyset$};
    			\draw[-] (6.7,2.4) -- (8.3,4) -- (9.9,2.4) -- (8.3,0.8) -- (6.7,2.4);
    			\draw[-] (8.3,0.8) -- (8.3,-0.8);
    			\filldraw (1.3,-1.7) node[below] {\small$(P,\succeq)$};
    			\filldraw (8.3,-1.7) node[below] {\small$\mathcal{J}(P)$};
    		\end{tikzpicture}
    	\end{center}
        Then, $H_P=(x_1x_2x_3,\,x_1x_2y_3,\,x_1y_2x_3,\,x_1y_2y_3,\,y_1y_2y_3)$.
    \end{Example}
    
    The Hibi ideal $H_P$ is equigenerated in degree $|P|$. It is well known that Hibi ideals have linear powers. Next, we calculate the $\v$-function of $H_P$.
    
    \begin{Theorem}\label{Thm:vNumHibi}
    	Let $H_P$ be a Hibi ideal. Then,
    	$$
    	\v(H_P^k)=k|P|-1,\ \ \textit{for all}\ \ k\ge1.
    	$$
    \end{Theorem}

    In order to prove the theorem, we need some preliminary lemmata.
    
    \begin{Lemma}\label{Lemma:AssHibiIdeal}
    	Let $H_P$ be a Hibi ideal. Then,
    	$$\Ass(H_P^k)\ =\ \{(x_i,y_j):p_i\preceq p_j\},\ \ \textit{for all} \ k\ge1.$$
    \end{Lemma}
    \begin{proof}
    	By \cite[Corollary 1.2]{HH2005a} the Rees algebra $\mathcal{R}(H_P)$ is a normal domain. Hence, $H_P$ is a normal ideal (see, also, \cite[Corollary 3.5]{CF2023b}). Thus, by \cite[Proposition (4.7)]{R1976}, $(H_P^{k+1}:H_P)=H_P^k$, for all $k\ge1$. By \cite[Theorem 9.1.13]{JT} $H_P$ is the cover ideal $J(G)$ of a Cohen--Macaulay bipartite graph $G$. Next, by \cite[Theorem 6.10]{VPV23} we have that $J(G)^k=J(G)^{(k)}$ for all $k\ge1$, that is, ordinary and symbolic powers of $J(G)=H_P$ coincide. Since $H_P$ is a squarefree monomial ideal, by \cite[Proposition 1.4.4 and Corollary 1.3.6]{JT} we have
    	$$
    	H_P^k\ =\ H_P^{(k)}\ =\ \bigcap_{p_i\preceq p_j}(x_i,y_j)^k.
    	$$
    	Hence, $\Ass(H_P^k)=\Ass(H_P)$ for all $k\ge1$.
    \end{proof}

    \begin{Lemma}
    	Let $H_P$ be a Hibi ideal. Then
    	\begin{enumerate}
    		\item[\textup{(a)}] $\v_{(x_i,y_i)}(H_P)=|P|-1$, for all $p_i\in P$.
    		\item[\textup{(b)}] $\v_{(x_i,y_j)}(H_P)=|P|+|\{p_\ell\in P:p_i\prec p_\ell\prec p_j\}|$, if $p_i\ne p_j\in P$ with $p_i\prec p_j$.
    	\end{enumerate}
    \end{Lemma}
    \begin{proof}
    	(a) Let $\mathcal{I}$ be the smallest poset ideal containing $p_i$. Hence, up to a suitable relabeling we may assume that $\mathcal{I}=\{p_1,\dots,p_{i-1},p_i\}$. It is clear that $\mathcal{I}\setminus\{p_i\}$ is again a poset ideal. Thus $x_1\cdots x_iy_{i+1}\cdots y_n, x_1\cdots x_{i-1}y_i\cdots y_n\in H_P$. This shows that $(H_P:x_1\cdots x_{i-1}y_{i+1}\cdots y_n)=(x_i,y_i)$. Hence $\v_{(x_i,y_i)}(H_P)\le|P|-1$. By Proposition \ref{Prop:Vnum>alpha}, we conclude that $\v_{(x_i,y_i)}(H_P)=|P|-1$, as wanted.
    	
    	(b) Let $\mathcal{I}_1,\mathcal{I}_2$ be the two smallest poset ideals of $P$ containing $p_i$, $p_j$, respectively. Since $p_i\prec p_j$, we have $\mathcal{I}_1\subset\mathcal{I}_2$. After a suitable relabeling, we may assume that $\mathcal{I}_1=\{p_1,\dots,p_{i-1},p_i\}$ and $\mathcal{I}_2=\{p_1,\dots,p_i,\dots,p_{j-1},p_j\}$. Consider the monomial
    	$$
    	u\ =\ x_1\cdots x_{i-1}y_i(\!\prod_{\ell=i+1}^{j-1}x_\ell y_\ell)x_{j}y_{j+1}\cdots y_n.
    	$$
    	Then, it is easily seen that $(H_P:u)=(x_i,y_j)$. Therefore,
    	$$
    	\v_{(x_i,y_j)}(H_P)\ \le\ \deg(u)=|P|+|\{p_\ell\in P:p_i\prec p_\ell\prec p_j\}|.
    	$$
    	To show the converse inequality, let $v$ be a monomial such that $(H_P:v)=(x_i,y_j)$ and $\deg(v)=\v_{(x_i,y_j)}(H_P)$. Then $v\notin H_P$, but $x_iv\in H_P$ and $y_jv\in H_P$. The definition of $H_P$ implies that $x_iv=w_1u_{\mathcal{J}_1}$ and $y_jv=w_2u_{\mathcal{J}_2}$, where $x_i$ divides $u_{\mathcal{J}_1}\in G(H_P)$ but not $v$, $y_j$ divides $u_{\mathcal{J}_2}\in G(H_P)$ but not $v$, and $w_1$ and $w_2$ are suitable monomials. Since $x_i$ does not divide $v$ but $y_jv\in H_P$, it follows that $y_i$ divides $v$. Hence $y_{i+1}\cdots y_{j-1}$ divides $v$, because if for some $\ell$, $y_\ell$ does not divide $v$, then $x_\ell$ should divide $y_jv$ and since $p_i\prec p_\ell$ this would imply that $x_i$ divides $v$, which is impossible. On the other hand, since $y_j$ does not divide $v$, but $x_iv\in H_P$, it follows that $x_j$ divides $v$, and in particular divides $u_{\mathcal{J}_1}$, hence $x_1\cdots x_i\cdots x_{j-1}$ divides $u_{\mathcal{J}_1}$, and so $x_1\cdots x_{i-1}x_{i+1}\cdots x_{j-1}$ divides $v$.
    	
    	Summarizing our reasoning, we have shown that $x_1\cdots x_{i-1}y_i(\prod_{\ell=i+1}^{j-1}x_\ell y_\ell)x_{j}$ divides $v$. Since $x_iv\in H_P$, a monomial $z$ of the type $z_{j+1}\cdots z_n$, with $z_\ell\in\{x_\ell,y_\ell\}$ for $\ell=j+1,\dots,n$, divides $v$. Hence $x_1\cdots x_{i-1}y_i(\prod_{\ell=i+1}^{j-1}x_\ell y_\ell)x_{j}z$ divides $v$ and so
    	\begin{align*}
    		\v_{(x_i,y_j)}(H_P)=\deg(v)\ &\ge\ \deg(x_1\cdots x_{i-1}y_i(\!\prod_{\ell=i+1}^{j-1}x_\ell y_\ell)x_{j}z)\ =\ n+(j-i-1)\\
    	&=\ |P|+|\{p_\ell\in P:p_i\prec p_\ell\prec p_j\}|,
    	\end{align*}
    	as desired.
    \end{proof}
    
    Let $(P,\succeq)$ be a finite poset, with $P=\{p_1,\dots,p_n\}$. Following \cite{CF2023}, for any integer $k\ge1$, we construct a new poset $(P(k),\succeq_k)$ defined as follows:
    \begin{enumerate}
    	\item[-] $P(k)\ =\ \{p_{i,1},p_{i,2},\dots,p_{i,k}\ :\ i=1,\dots,n\}$,
    	\item[-] $p_{i,r}\succeq_k p_{j,s}$ if and only if $p_i\succeq p_j$ and $r\ge s$.
    \end{enumerate}

    In order to preserve the structure of Hibi ideals, we innocuously modify polarization. Let $1\le \ell\le k$ and $i\in[n]$, then we set
    \begin{align*}
    	(x_i^\ell)^\wp\ &=\ x_{i,1}x_{i,2}\cdots x_{i,\ell},\\
    	(y_i^\ell)^\wp\ &=\ y_{i,\ell}y_{i,\ell-1}\cdots y_{i,k+1-\ell},
    \end{align*}
    and extend the polarization of an arbitrary monomial in the obvious way. In other words, with respect to the usual polarization, we are just applying the relabeling of the variables $y_{i,\ell}\mapsto y_{i,k+1-\ell}$ for $i=1,\dots,n$ and $\ell=1,\dots,k$.
    
    The next result was proved in \cite[Theorem 4.9]{CF2023}.
    \begin{Lemma}\label{Lemma:CrupiFicarraPolarization}
    	Let $(P,\succeq)$ be a finite poset. Then $(H_P^k)^\wp=H_{P(k)}$ for all $k\ge1$.
    \end{Lemma}
    \begin{Example}
    	\rm Consider the poset $(P,\succeq)$ of Example \ref{ex:posetVNum}. The poset $(P(2),\succeq_2)$ and the distributive lattice $\mathcal{J}(P(2))$ are depicted below. \medskip
    	
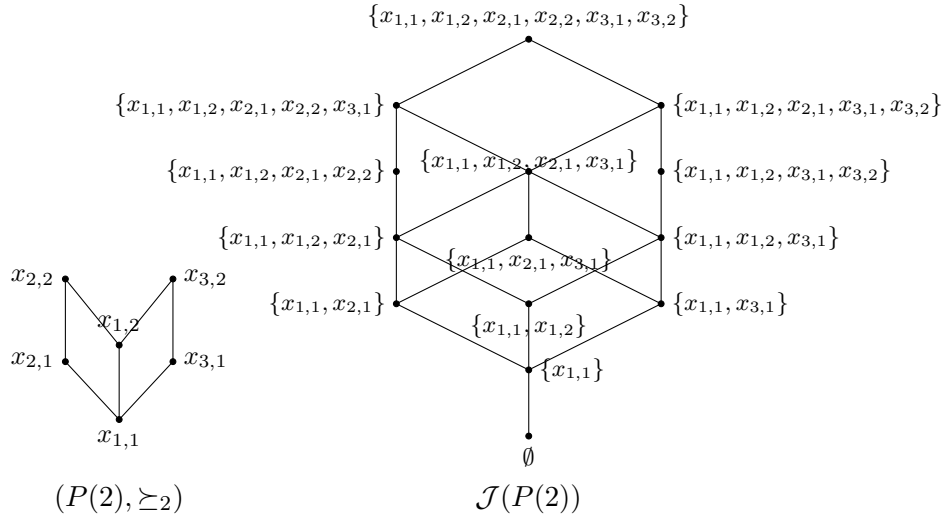
\begin{center}
    		\begin{tikzpicture}[scale=0.55]
    			\filldraw (5.3,-2.4) circle (2pt) node[below] {\footnotesize$x_{1,1}$};
    			\filldraw (4,-1) circle (2pt) node[left] {\footnotesize$x_{2,1}$};
    			\filldraw (6.6,-1) circle (2pt) node[right] {\footnotesize$x_{3,1}$};
    			\filldraw (5.3,-0.6) circle (2pt) node[above] {\footnotesize$x_{1,2}$};
    			\filldraw (4,1) circle (2pt) node[left] {\footnotesize$x_{2,2}$};
    			\filldraw (6.6,1) circle (2pt) node[right] {\footnotesize$x_{3,2}$};
    			\draw[-] (4,1) -- (5.3,-0.6) -- (6.6,1);
    			\draw[-] (4,1) -- (4,-1);
    			\draw[-] (5.3,-0.6) -- (5.3,-2.4);
    			\draw[-] (6.6,1) -- (6.6,-1);
    			\draw[-] (4,-1) -- (5.3,-2.4) -- (6.6,-1);
    			\filldraw (5.3,-3.7) node[below] {\small$(P(2),\succeq_2)$};
    			\filldraw (12,5.2) circle (2pt) node[left] {\scriptsize$\{x_{1,1},x_{1,2},x_{2,1},x_{2,2},x_{3,1}\}$};
    			\filldraw (18.4,5.2) circle (2pt) node[right] {\scriptsize$\{x_{1,1},x_{1,2},x_{2,1},x_{3,1},x_{3,2}\}$};
    			\filldraw (15.2,6.8) circle (2pt) node[above] {\scriptsize$\{x_{1,1},x_{1,2},x_{2,1},x_{2,2},x_{3,1},x_{3,2}\}$};
    			\filldraw (12,3.6) circle (2pt) node[left] {\scriptsize$\{x_{1,1},x_{1,2},x_{2,1},x_{2,2}\}$};
    			\filldraw (18.4,3.6) circle (2pt) node[right] {\scriptsize$\{x_{1,1},x_{1,2},x_{3,1},x_{3,2}\}$};
    			\filldraw (15.2,3.6) circle (2pt) node[above,yshift=-4.5] {\scriptsize$\{x_{1,1},x_{1,2},x_{2,1},x_{3,1}\}$};
    			\filldraw (12,2) circle (2pt) node[left] {\scriptsize$\{x_{1,1},x_{1,2},x_{2,1}\}$};
    			\filldraw (18.4,2) circle (2pt) node[right] {\scriptsize$\{x_{1,1},x_{1,2},x_{3,1}\}$};
    			\filldraw (15.2,2) circle (2pt) node[below] {\scriptsize$\{x_{1,1},x_{2,1},x_{3,1}\}$};
    			\filldraw (12,0.4) circle (2pt) node[left] {\scriptsize$\{x_{1,1},x_{2,1}\}$};
    			\filldraw (15.2,0.4) circle (2pt) node[below] {\scriptsize$\{x_{1,1},x_{1,2}\}$};
    			\filldraw (18.4,0.4) circle (2pt) node[right] {\scriptsize$\{x_{1,1},x_{3,1}\}$};
    			\filldraw (15.2,-1.2) circle (2pt) node[right] {\scriptsize$\{x_{1,1}\}$};
    			\filldraw (15.2,-2.8) circle (2pt) node[below] {\scriptsize$\emptyset$};
    			\draw[-] (15.2,-1.2) -- (18.4,0.4) -- (18.4,3.6);
    			\draw[-] (12,3.6) -- (12,0.4) -- (15.2,-1.2);
    			\draw[-] (12,2) -- (15.2,0.4);
    			\draw[-] (15.2,2) -- (15.2,3.6);
    			\draw[-] (15.2,0.4) -- (18.4,2) -- (15.2,3.6) -- (12,2);
    			\draw[-] (15.2,-2.8) -- (15.2,0.4);
    			\draw[-] (12,0.4) -- (15.2,2) -- (18.4,0.4);
    			\draw[-] (15.2,6.8) -- (18.4,5.2) -- (15.2,3.6) -- (12,5.2) -- (15.2,6.8);
    			\draw[-] (12,5.2) -- (12,3.6);
    			\draw[-] (18.4,5.2) -- (18.4,3.6);
    			\filldraw (15.2,-3.7) node[below] {\small$\mathcal{J}(P(2))$};
    		\end{tikzpicture}
    	\end{center}
        One can easily verify that $(H_P^2)^\wp=H_{P(2)}$, with respect to our modified polarization.
    \end{Example}
    Finally, Theorem \ref{Thm:vNumHibi} follows immediately from the next result.
    \begin{Corollary}\label{Cor:Hibi-V-Functions}
    	Let $H_P$ be a Hibi ideal and $k\ge1$ any integer. Then
    	\begin{enumerate}
    		\item[\textup{(a)}] $\v_{(x_i,y_i)}(H_P^k)=|P|k-1$, for all $p_i\in P$.
    		\item[\textup{(b)}] $\v_{(x_i,y_j)}(H_P^k)=|P|k+|\{p_\ell\in P:p_i\prec p_\ell\prec p_j\}|$, if $p_i\ne p_j\in P$ with $p_i\prec p_j$.
    	\end{enumerate}
    \end{Corollary}
    \begin{proof}
    	By Lemma \ref{Lemma:AssHibiIdeal}, $\Ass(H_P^k)=\{(x_i,y_j):p_i\preceq p_j\}$. Whereas, by Lemma \ref{Lemma:CrupiFicarraPolarization}, we have $(H_P^k)^\wp=H_{P(k)}$. Hence, by Theorem \ref{Theorem:v(I)v(I^wp)}(c) and Lemma \ref{Lemma:AssHibiIdeal},
    	\begin{align*}
    		\v_{(x_i,y_i)}(H_P^k)\ &=\ \min_{1\le \ell\le h\le k}\v_{(x_{i,\ell},y_{i,h})}(H_{P(k)}),\ \ \text{for all}\ p_i\in P,\ \text{and},\\
    		\v_{(x_i,y_j)}(H_P^k)\ &=\ \min_{1\le \ell\le h\le k}\v_{(x_{i,\ell},y_{j,h})}(H_{P(k)}),\ \ \text{for all}\ p_i\prec p_j.
    	\end{align*}
        A quick calculation shows that (a) and (b) indeed hold.
    \end{proof}
	
\end{document}